\documentclass[9pt,journal,twoside,web]{ieeecolor}
\usepackage{generic}

\usepackage{graphicx} 
\usepackage{graphics}
\usepackage{textcomp}
\usepackage{epsfig}
\usepackage{times} 
\usepackage{amsmath}
\usepackage{amssymb}
\usepackage{amsfonts}
\usepackage{float}
\usepackage{siunitx}
\usepackage{scalerel}
\usepackage{cite}
\newtheorem{theorem}{Theorem}
\newtheorem{proposition}{Proposition}
\newtheorem{lemma}{Lemma}

\newtheorem{remark}{Remark}
 \newtheorem{assumption}{Assumption}
 \newtheorem{Corollary}{Corollary}

\title{\LARGE \bf Distributed Estimation}
\newcommand{\ncom}{\newcommand}

\newcommand{\beqn}{\begin{eqnarray*}}
\newcommand{\eeqn}{\end{eqnarray*}}
\newcommand{\beq}{\begin{eqnarray}}
\newcommand{\eeq}{\end{eqnarray}}

\newcommand{\norm}[1]{\left\lVert #1 \right\rVert}
\newcommand{\inprod}[2]{\left\langle #1, #2 \right\rangle}
\ncom\R{\mathbb{R}}
\DeclareMathOperator*{\argmin}{arg\,min}

\DeclareMathOperator*{\dist}{dist}
\author{Anik Kumar Paul and Shalabh Bhatnagar
\thanks{The authors are with the Department of Computer Science and Automation, Indian Institute of Science, Bengaluru 560012, India.
        {email: anikpaul42@gmail.com, shalabh@iisc.ac.in.}}
\thanks{This work was supported in part by the Walmart Centre for Tech Excellence, IISc, through Project No.~DFTM/02/3125/M/04/AIR-04 from DRDO under DIA-RCOE, by a J.C.Bose Grant with No.ANRF/JBG/2025/000209/HAA from ANRF, Government of India, and by the B.S.Sonde Chair Professorship from IISc.}
    }

\begin{document}
\title{Zeroth-Order Non-smooth Non-convex Optimization via Gaussian Smoothing}
\maketitle
\thispagestyle{plain}
\pagestyle{plain}
\begin{abstract}
    This paper addresses stochastic optimization of Lipschitz-continuous, non-smooth and non-convex objectives over compact convex sets, where only noisy function evaluations are available. While gradient-free methods have been developed for smooth non-convex problems, extending these techniques to the non-smooth setting remains challenging. The primary difficulty arises from the absence of a Taylor series expansion for Clarke subdifferentials, which limits the ability to approximate and analyze the behavior of the objective function in a neighborhood of a point. We propose a two time-scale zeroth-order projected stochastic subgradient method leveraging Gaussian smoothing to approximate Clarke subdifferentials. First, we establish that the expectation of the Gaussian-smoothed subgradient lies within an explicitly bounded error of the Clarke subdifferential, a result that extends prior analyses beyond convex/smooth settings. Second, we design a novel algorithm with coupled updates: a fast timescale tracks the subgradient approximation, while a slow timescale drives convergence. Using continuous-time dynamical systems theory and robust perturbation analysis, we prove that iterates converge almost surely to a neighborhood of the set of Clarke stationary points, with neighborhood size controlled by the smoothing parameter. To our knowledge, this is the first zeroth-order method achieving almost sure convergence for constrained non-smooth non-convex optimization problems. 
\end{abstract}

\section{Introduction}
We consider the following stochastic optimization problem:
\begin{equation}
    \min_{x \in \mathcal{X}} f(x) := \mathbb{E} [F(x,\zeta)] = \int F(x,\zeta) \, d\mathbb{P}(\zeta),
    \label{eq:main_problem}
\end{equation}
where $\mathcal{X} \subseteq \mathbb{R}^d$ is a compact and convex decision set and $F: \mathbb{R}^d \times \Omega \to \mathbb{R}$ is a measurable function. We assume $f$ is Lipschitz continuous, but we do not impose convexity or differentiability assumptions. 

The problem in \eqref{eq:main_problem} has been widely studied in stochastic optimization \cite{ermoliev2003solution,bolte2021conservative,xu2019non,mai2020convergence,lin2022gradient}, with important applications in machine learning and statistical learning where objective functions are often non-smooth and non-convex \cite{bolte2021conservative}. Since classical stochastic gradient methods require differentiability, several approaches replace the gradient with elements of the Clarke subdifferential to enable iterative optimization in the non-smooth setting \cite{xu2019non,mai2020convergence}. These works establish $\mathcal{L}^1$ convergence to the set of Clarke stationary points, i.e., points where zero belongs to the subdifferential, under weak convexity assumptions. Moreover, \cite{bianchi2022convergence} proves almost sure convergence of a stochastic subgradient method for  non-smooth non-convex functions.

However, these approaches critically rely on access to an oracle providing noisy estimates of the Clarke subdifferential, which poses two major challenges. First, in many practical settings, particularly simulation-based optimization, only noisy measurements of the objective function are available, making it necessary to construct estimators of the  Clarke subgradients from these  \cite{prashanth2025gradient}. Second, extending standard gradient-based methods to the non-smooth setting through the Clarke subdifferential is nontrivial due to the absence of a complete chain rule. Moreover, computing the Clarke subdifferential can be computationally expensive in large-scale problems \cite{Nesterov2015RandomGM}.

There has been extensive work on zeroth-order optimization for non-convex objectives \cite{spall1992,B2007,SPP2013,ghadimi}, typically under the assumption that the objective function is continuously differentiable with Lipschitz continuous gradients. These methods employ random perturbations, e.g., Gaussian perturbations, to estimate gradients \cite{B2007,SPP2013,Nesterov2015RandomGM,ghadimi}. Almost sure convergence of the iterates is established in \cite{B2007,SPP2013,prashanth2025gradient,Ramaswami}, while \cite{Nesterov2015RandomGM,ghadimi} derive finite-time bounds.

Zeroth-order optimization for non-convex non-smooth functions has also been studied in \cite{lin2022gradient,chen2023faster,kornowski2024algorithm,liu2024zeroth}. However, these works mainly derive upper bounds on the $\mathcal{L}^1$ norm of the minimum Goldstein subgradient by applying stochastic gradient descent type arguments to smooth approximations of the underlying objective. In contrast, optimality for non-convex, non-smooth problems is naturally characterized through the Clarke subdifferential \cite{rockafellar1998variational}, while the Goldstein subdifferential is known to be weaker than the Clarke subdifferential \cite{lin2022gradient}.
\color{black}

Motivated by the broad range of non-smooth non-convex optimization problems arising in deep learning and machine learning \cite{bolte2021conservative}, this paper addresses the following questions: 1) whether Gaussian smoothing can approximate the Clarke subdifferential of a non-convex Lipschitz continuous function and what properties such an approximation satisfies, and 2) whether one can design a zeroth-order algorithm for \eqref{eq:main_problem} whose iterates converge almost surely to the set of Clarke stationary points. We provide affirmative answers to both questions. To the best of our knowledge, this is the first work establishing almost sure convergence of a zeroth-order method for non-convex Lipschitz continuous functions. 
\color{black}

We show in this paper that the expectation of the subgradient approximation obtained through Gaussian smoothing can be represented as a Clarke subgradient together with an error term, where the error can be made arbitrarily small by an appropriate choice of the smoothing parameter.

Our analysis adopts a dynamical systems perspective similar to \cite{bianchi2022convergence,bianchi2003solution}, where almost sure convergence to Clarke stationary points is established via the stability analysis of a differential inclusion. However, the present work differs from \cite{bianchi2022convergence,bianchi2003solution} in two important aspects. First, we consider the constrained optimization problem \eqref{eq:main_problem}. A direct extension of \cite{bianchi2022convergence} using Gaussian-based subgradient approximations together with projection onto the compact constraint set encounters problems discussed in Remark \ref{Remark-two time scale}. To overcome this, we develop a two-timescale stochastic subgradient method and analyze its asymptotic behavior through a projected differential inclusion. Second, unlike \cite{bianchi2003solution,bianchi2022convergence}, the algorithm does not have access to exact Clarke subgradients due to subgradient approximation, resulting in an additional error term in the dynamics. Consequently, convergence is established through a robust stability analysis of a projected differential inclusion involving the Clarke subdifferential with disturbance input.

\color{black}

 The main contributions of this paper are summarized as follows. 

\begin{enumerate}
    \item We propose a two time-scale zeroth-order projected stochastic subgradient method to solve the constrained optimization problem in \eqref{eq:main_problem}. The method employs Gaussian smoothing to approximate the subdifferential of the objective function.
    \item We establish that the expectation of the approximated subgradient corresponds to an element of the Clarke subdifferential, up to a nonzero bias that can be controlled through the choice of the smoothing parameter.
    \item We prove that the iterations of the proposed algorithm converge almost surely to a neighborhood of the set of Clarke stationary points, with the diameter of this neighborhood determined by the bias introduced by the Gaussian approximation.
\end{enumerate}

\section{Notation}

We let $\mathbb{R}^d$ denote the $d$-dimensional Euclidean space with inner product $\langle \cdot, \cdot \rangle$ and induced norm $\| \cdot \|$. For a compact convex set $\mathcal{X} \subset \mathbb{R}^d$, $\mathcal{P}_{\mathcal{A}}(x)$ denotes the Euclidean projection of $x$ onto $\mathcal{X}$, and $\text{dist}(x, \mathcal{X}) := \min_{y \in \mathcal{X}} \|x - y\|$. The closed ball of radius $r$ centered at the origin is denoted by $B(0, r)$. The normal and tangent cones to $\mathcal{X}$ at $x \in \mathcal{X}$ are denoted by $\mathcal{N}_{\mathcal{X}}(x)$ and $\mathcal{T}_{\mathcal{X}}(x)$, respectively. For a locally Lipschitz function $f$, $\partial f(x)$ shall denote the Clarke subdifferential at $x$ and $f^0(x,v)$ the generalized directional derivative (see \cite{clarke1990optimization}). Throughout the paper, we assume the existence of $G>0$ such that $\|g\|\leq G$ for all $g\in \partial f(x)$ and all $x\in\mathcal{X}$.


\section{Zeroth-Order Two Time-Scale Projected Stochastic Subgradient Method}

We now present the zeroth-order stochastic subgradient algorithm to solve the optimization problem in \eqref{eq:main_problem}. A key aspect of our setting is that we assume that while we do not know the form of the function $f(\cdot)$, we have access to an oracle that returns a noisy function value $F(z,\zeta)$  for any input $z \in \mathbb{R}^d$. The steps of the algorithm are outlined below. We present a two-timescale zeroth-order projected stochastic subgradient method to solve the problem \eqref{eq:main_problem}.

Let $x_n$, $y_n$  denote the iterates of the algorithm at recursion $n$. 
To approximate the subdifferential of the function $f$ at $x=x_n$, we use a zeroth-order oracle that returns the function values $F(x_n+ \lambda_n U_n, \zeta^1_n)$ and $F(x_n - \lambda_n U_n, \zeta^2_n )$,  where $\lambda_n > 0$ is a given smoothing parameter and for $n\geq 0$, $U_n \sim \mathcal{N}(0,I)$ are a sequence of independent standard Gaussian vectors. Based on these values, we construct the following subgradient estimator:

\begin{equation}
     \Tilde{g} (n)  = \left(\frac{F(x_n + \lambda_n U_n, \zeta_n^1) - F(x_n-\lambda_n U_n ,\zeta_n^2)}{2 \lambda_n}\right) U_n.
     \label{approx}
 \end{equation}

The next iterate is then computed using the projected stochastic subgradient update, where the true gradient is replaced by the estimator  $\Tilde{g}(n)$. Specifically, the update rule is given by:

\begin{equation*}
    y_{n+1} = y_n + \beta(n) (\Tilde{g}(n)-y_n),
\end{equation*}
 \begin{equation}
     x_{n+1} = \mathcal{P}_\mathcal{X} \Big{(} x_n - \alpha(n) y_n \Big{)}. 
     \label{mirror descent}
 \end{equation}
 Here $\alpha(n),\beta(n) > 0$, $\forall n$, denote the step-sizes that satisfy the conditions in Assumption~\ref{ass-stepsize} and $\mathcal{P}_\mathcal{X}(a)$ is the Euclidean projection of $a\in\mathbb{R}^d$ onto the set $\mathcal{X}$.
The objective of this paper is to analyze the behavior of the algorithm's iterates and establish  asymptotic guarantees, in particular, almost sure convergence. We shall make the assumptions listed below. 
\begin{assumption}
    The variance of the stochastic objective function $F(x_n,\zeta)$ is uniformly bounded.  Specifically, $\exists \; \mathrm{K} > 0$ such that\\
$
    E[(F(x_n,\zeta)-f(x_n))^2| x_n] \leq \mathrm{K}.
$
\label{assumpvari}
\end{assumption}
Let $\mathcal{F}_n = \sigma(x_m, U_m, m\leq n, \zeta^1_m, \zeta^2_m, m<n)$, $n\geq 1$, denote the associated filtration. Then, 
Assumption \ref{assumpvari} implies that $F(\cdot,\cdot)$ admits the following decomposition: $ F(x_n + \lambda_n U_n, \zeta_n^1) = f(x_n+ \lambda_n U_n) + N_{n+1}^1, \quad \text{and}$
$  F(x_n - \lambda_n U_n, \zeta_n^2) = f(x_n - \lambda_n U_n) + N_{n+1}^2, \quad \text{respectively,}$
 where $(N_{n}^1,\mathcal{F}_n)$, $(N_{n}^2,\mathcal{F}_n)$, $n\geq 0$, are two martingale difference sequences satisfying $ \mathbb{E}[(N_{n+1}^i)^2 | \mathcal{F}_n] \leq \mathrm{K}, \quad \forall \; i \in \{1,2\}.$

\begin{assumption}
   The step-sizes $\alpha(n),n\geq 0$ and $\beta(n), n\geq 0$ satisfy the following conditions: 
    \begin{itemize}
        \item[(a)] $\alpha(0) < 1$ and $\forall \; n \geq 0$, $\alpha(n) > \alpha(n+1)$.
          \item[(b)] $\beta(0) < 1$ and $\forall \; n \geq 0$, $\beta(n) > \beta(n+1)$.
          \item[(c)] $\sum\limits_{n \geq 0} \alpha(n)$ $=$ $\sum\limits_{n \geq 0} \beta (n)$ $=$ $\infty$,  $\sum\limits_{n \geq 0} (\alpha(n)^2 + \beta(n)^2) < \infty$.
          \item[(d)] $\lim\limits_{n \to \infty} \frac{\alpha(n)}{\beta(n)} = 0$.
          
    \end{itemize}
    \label{ass-stepsize}
\end{assumption}

\color{black}

In this paper, we consider two different choices for the smoothing parameter sequence $\{\lambda_n\}$. The first case corresponds to a fixed smoothing parameter, where $\lambda_n = \lambda > 0.$
The second case considers a diminishing sequence $\lambda_n \to 0$ deterministically as $n \to \infty$. For this setting, we impose the following assumption.

\begin{assumption}
The smoothing parameter sequence $\{\lambda_n\}$ satisfies
$ \lim_{n\to\infty} \lambda_n = 0,
    \quad
    \sum_{n\geq 1} \frac{\beta_n^2}{\lambda_n^2} < \infty.$
\label{Assumptionregardingsmoothing}
\end{assumption}
In the case of a constant smoothing parameter, we require the following additional assumption.
\color{black}
\begin{assumption}
    The subdifferential mapping of 
$f$ is assumed to satisfy a one-sided Lipschitz condition. Specifically, for any $x,y  \in \mathcal{X}$ and any $g_x \in \partial f(x)$, $g_y \in \partial f(y)$, the following holds: 
    \begin{equation*}
        (g_x-g_y)^\top (y-x) \leq L \norm{x-y}^2, \quad \text{for some  $L > 0$.}
    \end{equation*}
    \label{monotone}
\end{assumption}

\begin{remark}
\label{assum3-4}
Assumption~\ref{monotone} is standard in non-convex optimization and is satisfied by most functions in machine learning. It covers continuously differentiable functions with Lipschitz gradients, convex functions, and extends to weakly convex functions widely used in non-smooth and non-convex optimization \cite{davis2019stochastic}, particularly in deep learning \cite{zhou2017landscape}. For example, functions of the form $f(x) = \max_{1 \leq i \leq N} f_i(x)$, where each $f_i$ has a Lipschitz continuous gradient, naturally arise in robust training or adversarial settings, where the worst-case objective is non-convex and non-smooth. Another instance is composite functions $f(x) = g(h(x))$, with $h$ differentiable and having an $L$-Lipschitz Jacobian, and $g$ Lipschitz continuous, a structure common in learning applications \cite{mai2020convergence}. In neural networks, with input vectors ${u_i}$, outputs ${v_i}$, parameters $x = (a, W)$, and smooth non-convex activations such as $\tanh$, the empirical loss
  \begin{equation*}
    f(x) = \frac{1}{N} \sum\limits_{i=1}^{N} f_i(x) = \frac{1}{N} \sum\limits_{i=1}^{N} |v_i - a^\top \sigma (W u_i)|.
\end{equation*}
satisfies Assumption~\ref{monotone}.
\end{remark}
In the next section, we show that the expected approximated subgradient in \eqref{approx} equals a Clarke subgradient up to an error that diminishes with $\lambda$, and that the iterates $\{x_n\}$ almost surely converge to a $\lambda$-dependent neighborhood of the Clarke stationary set.

\section{Almost Sure Convergence of the Algorithm}
We begin by establishing key properties of the approximated subgradient, which  serve as a cornerstone for the convergence guarantees developed in the sequel.

\subsection{Properties of Approximated subgradient}

\begin{theorem}
    Let $\mathcal{F}_n$ denote the sigma-algebra generated by $x_1,\ldots,x_n$, that is, $ \mathcal{F}_n = \sigma ( \; x_k \; | \; 1 \leq k \leq n \; ). \quad \text{Then} $
\begin{equation}
        \mathbb{E}[\Tilde{g}(n) | \mathcal{F}_n] \in \partial f(x_n) + B(0,r(\lambda)),
        \label{expectation}
    \end{equation}
   where $\lim\limits_{\lambda \to 0} r (\lambda) = 0$. Also we have 
\begin{equation}
        \mathbb{E}[\norm{\Tilde{g}(n)}^2 | \mathcal{F}_n] \leq 2 L^2 (d^2 + d)  + \frac{\mathrm{K}^2}{\lambda^2} d. 
        \label{variance}
    \end{equation}
    \label{property}
\end{theorem}
The proof of the main theorem relies on the following technical lemma, which we state and prove below. 

\begin{lemma}
    Let $\lambda > 0$,  and define the Gaussian smoothed version of the function $f$ as  
    $
        f_\lambda (x) = \mathbb{E}_u [f(x+\lambda u)], 
    $
where $u \sim \mathcal{N}(0,I)$.
    Then, there exists a function $r : \mathbb{R}_{+} \to \mathbb{R}_{+} $ satisfying $\lim\limits_{\lambda \to 0} r(\lambda) = 0$ such that
    $
        \nabla f_\lambda (x) \in \partial f(x) + B(0,r(\lambda)).
    $
    \label{conv}
\end{lemma}

\begin{proof}
     Note that the function  $f_\lambda : \mathcal{X}  \to \mathbb{R}$  is defined as
    \begin{equation*}
        f_\lambda(x) = \mathbb{E}_u [f(x+ \lambda u)]. 
    \end{equation*}

  We now prove the lemma in a sequence of steps below. 
  
    \textbf{Step-1:} It is well known from \cite{Nesterov2015RandomGM} that $f_\lambda$ is differentiable because $f$ is  Lipschitz with respect to $x$.  In this step, we show that the gradient of the smoothed function $f_\lambda$
  can be expressed as
    \begin{equation*}
         \nabla f_\lambda (x)  = \mathbb{E}_u [\nabla f(x+\lambda u)].
    \end{equation*}

   Consider the partial derivative with respect to the $i$-th coordinate of $x$. Using the definition of directional derivative:
    \begin{equation}
        \begin{split}
            \frac{\partial f_\lambda (x)}{\partial x_i} & = \lim\limits_{h \downarrow 0} \frac{f_\lambda(x+h e_i )-f_\lambda(x)}{h}
            \\ & = \lim\limits_{h \downarrow 0} \mathbb{E}_u \Big{[} \frac{ f(x+\lambda u + h e_i) -f(x+\lambda u)}{h}  \Big{]}
            \\ & =\mathbb{E}_u \Big{[} \lim\limits_{h \downarrow 0} \frac{ f(x+\lambda u + h e_i) -f(x+\lambda u)}{h}  \Big{]} .
        \end{split}
        \label{dct}
    \end{equation}
    The interchange of limit and integration holds in view of the Dominated Convergence Theorem. Note that since $f$ is  Lipschitz,  we have 
    \begin{equation*}
        \left|\frac{ f(x+\lambda u + h e_i) -f(x+\lambda u)}{h}\right| \leq L.
    \end{equation*}

\color{black}
Note that $f$ may not be differentiable everywhere. However, by Rademacher’s theorem, $f$ is differentiable almost everywhere. Since $u \sim \mathcal{N}(0,I)$ and the Gaussian measure is absolutely continuous with respect to the Lebesgue measure, the set of non-differentiability points has Gaussian measure zero \cite{durrett2019probability}. Therefore, the last equality in \eqref{dct} holds. By the same argument, we also obtain
$
\nabla f_\lambda(x)=\mathbb{E}_u[\nabla f(x+\lambda u)].
$
\color{black}

\textbf{Step-2:} In this step, we show that
\begin{equation*}
    \nabla f_\lambda (x) \in \partial f(x) + B(0,r(x,\lambda)),
\end{equation*} 
such that for each $x$, $\lim\limits_{\lambda \to 0} r(x,\lambda) = 0$.
Since the function $f$ is assumed to be $L$-Lipschitz, we have $ \norm{\nabla f(x+\lambda u)} \leq L \; \; \text{a.e.} $ 
Then from Step 1 we conclude that $\norm{\nabla f_\lambda (x)} \leq L$. 

Let $\{\lambda_n\}_{n \geq 1}$
  be a sequence such that  $\lambda_n \to 0$ as 
$n \to \infty$, and suppose that the limit $\lim\limits_{n \to \infty} \nabla_x f_{\lambda_n}(x) $ exists. From the identity established in Step 1, we have $ \nabla f_{\lambda_n} (x) = \mathbb{E}_u [\nabla f(x+\lambda_n u)].$

  Assume, without loss of generality, that $\lim\limits_{n \to \infty} \nabla_x f(x+\lambda_n u)  \to \psi (u)$ pointwise a.e., or otherwise there exists a convergent subsequence. From the definition of subdifferential, we conclude that $ \psi (u) \in \partial f(x)  \; \; \text{a.e.}$


Taking the limit as  $n \to \infty$  and applying the Dominated Convergence Theorem, we obtain
\begin{equation*}
  \lim\limits_{n \to \infty}  \nabla_x f_{\lambda_n} (x) = \mathbb{E}_u [ \lim\limits_{n \to \infty} \nabla_x f(x+\lambda_n u)] = \mathbb{E}_u[\psi(u)].
\end{equation*}

Since $\partial f(x)$ is convex, it follows that $\mathbb{E}_u[\psi(u)] \in \partial f(x),$ 
which implies $\lim\limits_{n \to \infty} \nabla f_{\lambda_n} (x) \in \partial f(x).$ 
Consequently, we can show that $\lim\limits_{\lambda \to 0} \dist (\nabla f_\lambda (x), \partial f(x))
    =    0$.

Suppose, by contradiction, that the above limit does not hold. Then, $\exists \; \epsilon >0$ such that $\forall \; \delta > 0$, $\exists \; \lambda < \delta$ for which $ \dist (\nabla f_\lambda (x), \partial f(x)) > \epsilon.$

Now, define a sequence $\delta_n = \frac{1}{n}$, $n\geq 1$. By the assumption, for each $n$,  $\exists \; \lambda_n \leq \delta_n $ such that  $ \dist (\nabla f_{\lambda_n} (x), \partial f(x)) > \epsilon. $

Since $\nabla f_{\lambda_n} (x)$ is bounded, it admits a convergent subsequence. Let $\nabla f_{n_k} (x) \to v$ along a subsequence $\{n_k\}\subset \{n\}$. That implies $v \in \partial f(x)$, which is a contradiction. 

\textbf{Step-3} In this step, we further show that the function 
$r$ can be chosen independently of 
$x$; that is, there exists a function 
$\Bar{r}(\lambda)$ such that --- $ \nabla f_\lambda  (x)  \in \partial f(x) + B(0,\Bar{r}(\lambda)).$
with  $\lim\limits_{\lambda \to 0} \Bar{r}(\lambda) = 0$.

From the result of Step 2, we have 
\begin{equation*}
\begin{split}
    r(x, \lambda) & = \dist (\nabla f_\lambda(x), \partial f(x))
    = \min\limits_{g \in \partial f(x)} \norm{g - \nabla f_\lambda(x)}^2.
    \end{split}
\end{equation*}
We aim to show that  $r(x,\lambda)$ is upper semicontinuous in $x$ for each fixed $\lambda > 0$. 
Let $x_n,n\geq 0$, be a sequence such that $x_n \to x$. Then,
\begin{equation}
    \begin{split}
        & r(x_n,\lambda)  =   \dist (\nabla f_\lambda(x_n), \partial f(x_n))
     \leq   \dist (\nabla f_\lambda(x_n), \nabla f_\lambda(x)) \\ & \qquad + \dist(\nabla f_\lambda (x), \partial f(x)) + \dist (\partial f(x),\partial f(x_n)). 
    \end{split}
    \label{uscs}
\end{equation}
Notice that  $ \lim\limits_{n \to \infty} \dist (\nabla f_\lambda(x_n), \nabla f_\lambda(x)) = 0$ since $f_\lambda(x)$ is continuously differentiable. Also, since the subdifferential map is semicontinuous, $ \lim\limits_{n \to \infty} \dist (\partial f(x),\partial f(x_n))  = 0.$

Taking $\limsup$ in \eqref{uscs} gives $\limsup\limits_{n \to \infty} r(x_n,\lambda) \leq r(x,\lambda) $
showing that 
$r(x,\lambda)$ is upper-semicontinuous in $x$ for each fixed 
$\lambda$. Hence, over the compact set $\mathcal{X}$, $r(x,\lambda)$ attains its maximum; define $\bar{r}(\lambda) := \max_{x \in \mathcal{X}} r(x,\lambda)$. From the previous discussion, $\lim_{\lambda \to 0} \bar{r}(\lambda) = 0$.
\end{proof}

\textbf{Proof of Theorem \ref{property}}

\begin{proof}
From equation (21) in \cite{Nesterov2015RandomGM}, it follows that
{\small 
\begin{equation*}
    \begin{split}
        \mathbb{E}[\Tilde{g}(n) | \mathcal{F}_n] & = \mathbb{E}\Big{[}\frac{f(x_n+ \lambda U_n)-f(x_n-\lambda U_n)}{2 \lambda} U_n | \mathcal{F}_n \Big{]} = \nabla f_\lambda (x_n).
    \end{split}
\end{equation*}
}
Using Lemma \ref{conv}, we immediately obtain the conclusion stated in \eqref{expectation}. To show \eqref{variance} we consider the following: 
\begin{equation*}
    \begin{split}
       & \mathbb{E} [\norm{\Tilde{g}(n)}^2 | \mathcal{F}_n]
       \\ = & \mathbb{E} \Big{[} \norm{\frac{f(x_n+ \lambda U_n) + N_{n+1}^1- f(x_n- \lambda U_n) - N_{n+1}^2}{2 \lambda} U_n}^2| \mathcal{F}_n \Big{]}
       \\ \leq & \mathbb{E} \Big{[} \frac{(f(x_n+ \lambda U_n) - f(x_n- \lambda U_n))^2 }{2 \lambda^2} \norm{U_n}^2| \mathcal{F}_n \Big{]}
       \\ & + \mathbb{E} \Big{[}  \frac{ (N_{n+1}^1-   N_{n+1}^2)^2}{2 \lambda^2} \norm{ U_n}^2| \mathcal{F}_n \Big{]}
       \\ \leq  & 2 L^2 \mathbb{E}[\norm{U_n}^4] +  \frac{\mathrm{K}^2}{\lambda^2} \mathbb{E}[\norm{U_n}^2]
        \leq   2 L^2 (d^2 + d)  + \frac{\mathrm{K}^2}{\lambda^2} d.
    \end{split}
\end{equation*}
\color{black}
The last inequality follows from standard bounds on the second and fourth moments of a standard Gaussian random vector; see \cite{Nesterov2015RandomGM}.
\color{black}
\end{proof}
\begin{remark}
Theorem \ref{property} extends Nesterov’s Gaussian smoothing to non-convex non-smooth functions and forms the basis of our convergence analysis. It shows that the expected approximated subgradient converges to a Clarke subgradient as $\lambda \to 0$, while \eqref{variance} characterizes the associated bias--variance trade-off.
\end{remark}
\color{black}
A direct consequence of Theorem \ref{property} can be obtained when the smoothing parameter $\lambda_n$ is chosen to be diminishing. In this case, the approximation error in the subgradient estimate vanishes asymptotically as $n \to \infty$.

\begin{Corollary}
Under the same assumptions as in Theorem \ref{property}, suppose that the smoothing parameter $\lambda_n$ is diminishing and satisfies Assumption \ref{Assumptionregardingsmoothing}. Then, the approximated subgradient satisfies
\begin{equation}
    \mathbb{E}[\Tilde{g}(n)\mid \mathcal{F}_n]
    \in \partial f(x_n) + B(0,r(\lambda_n)),
    \label{expectationc}
\end{equation}
where $\lim_{n\to\infty} r(\lambda_n)=0.$
Moreover, the second moment of the approximated subgradient admits the bound
\begin{equation}
    \mathbb{E}\!\left[\norm{\Tilde{g}(n)}^2 \mid \mathcal{F}_n \right]
    \leq
    2L^2(d^2+d)
    +
    \frac{\mathrm{K}^2 d}{\lambda_n^2}.
    \label{variancec}
\end{equation}
\label{Corlloarydiminishinglambda}
\end{Corollary}

\color{black}

The following corollary follows directly from Theorem \ref{property} and Corollary \ref{Corlloarydiminishinglambda}.

\begin{Corollary}
    The approximate subgradient $\Tilde{g}(n)$ admits the decomposition
    \begin{equation*}
        \Tilde{g}(n) = g(n) + \mathrm{B} (n) + M_{n+1},
    \end{equation*}
    where $g(n) \in \partial f(x_n)$,  $\mathrm{B}(n)$ is the bias term satisfying $\norm{\mathrm{B}(n)} \leq r(\lambda_n)$, and $M_{n},n\geq 0$ is the martingale difference sequence adapted to the filtration $\{\mathcal{F}_n \}$ and satisfying $ \mathbb{E}[\norm{M_{n+1}}^2 | \mathcal{F}_n] \leq \mathrm{V_n},$ 
    where  $ \mathrm{V}_n = 6L^2 (d^2 +d) + 3 \frac{\mathrm{K}^2}{\lambda_n^2} d + 3 G^2 + 3 r(\lambda_n)^2.$
    \label{nosie decomposition}
\end{Corollary}

\begin{Corollary}
Under Assumptions \ref{ass-stepsize} and \ref{Assumptionregardingsmoothing}, the following holds:  
    \begin{equation*}
   \lim\limits_{n \to \infty} \sup\limits_{n \leq k \leq \tau^1(n,T)} \norm{\sum\limits_{m=n}^{k} \beta(m)  M_{m+1}} = 0, \quad \text{where}
\end{equation*}
 \begin{equation*}
        \tau^1(n,T) = \min \left\{ m \geq n \; | \sum\limits_{k=n}^{m+1}  \beta(k) \geq T \right\}.
    \end{equation*}
    \label{Corolloary assumption on noise}
\end{Corollary}

\begin{proof}
    Define  $  S_n = \sum\limits_{k=1}^{n} \beta(k) M_{k+1}.$ 
 Then the sequence ${S_n},n\geq 0,$ is a martingale that is bounded in $L^2$ as 
    \begin{equation*}
        \mathbb{E}[S_n^2] = \sum\limits_{k=1}^{n} \mathbb{E} \beta(k)^2 [\norm{M_{k+1}}^2] \leq  \sum\limits_{k=1}^{\infty} \beta(k)^2 \mathrm{V} < \infty.
    \end{equation*}
    \color{black}
The last inequality holds true in view of Assumption \ref{Assumptionregardingsmoothing}.
    \color{black}
   By the martingale convergence theorem, $\{S_n\}$ converges almost surely. Consequently, its tail sequence  $T_n = \sum\limits_{k=n}^{\infty} \beta(k) M_{k+1}, n\geq 0,$
    converges to zero almost surely. Therefore, for any $k \geq n$,
\begin{equation*}
\begin{split}
   \sum\limits_{m=n}^{k} \beta(m) M_{m+1} = T_n -  \sum\limits_{m=k+1}^{\infty} \beta(m) M_{m+1}.
    \end{split}
\end{equation*}
Taking norms and using the definition of $T_m$, we have
\begin{equation*}
\begin{split}
   & \sup\limits_{n \leq k \leq \tau^1(n,T)} \norm{\sum\limits_{m=n}^{k} \beta(m) M_{m+1}}
    \leq  \norm{T_n} + \sup\limits_{m \geq n} \norm{T_m}.
    \end{split}
\end{equation*}
Finally, letting $n \to \infty$, and noting that $T_n \to 0$ almost surely, yields
\begin{equation*}
   \lim\limits_{n \to \infty} \sup\limits_{n \leq k \leq \tau^1(n,T)} \norm{\sum\limits_{m=n}^{k} \beta(m)  M_{m+1}} = 0.
\end{equation*}
\end{proof}
\color{black}
Note that Corollaries \ref{nosie decomposition} and \ref{Corolloary assumption on noise} continue to hold for the case of a constant smoothing parameter. Since the proof follows along similar lines, we omit the details for brevity.

We require the following corollary to apply the stochastic approximation framework on two timescales developed in \cite{yaji2020stochastic}.

\begin{Corollary}
If the conclusions of Corollary \ref{nosie decomposition} and Corollary \ref{Corolloary assumption on noise} hold, then the iterate sequence $\{y_n\}$ is almost surely bounded.
\end{Corollary}

\begin{proof}
Note that, from the iterate updates in \eqref{mirror descent} and in view of Corollary \ref{nosie decomposition}, we have
 \begin{equation*}
    y_{n+1}
    =
    (1-\beta(n))y_n
    +
    \beta(n)
    \big(
    U_n + M_{n+1}
    \big).
\end{equation*}
where $U_n = g(n) + B(n).$
Since both $g(n)$ and $B(n)$ are uniformly bounded, it follows that there exists  $C>0$ such that $\norm{U_n} \leq C.$

By applying a standard recursive expansion, we obtain
\begin{equation}
    y_{n+1}
    =
    \phi(n,0)y_0
    +
    \sum_{k=0}^{n-1}
    \phi(n,k+1)\,
    \beta(k)
    \big(
    U_k + M_{k+1}
    \big),
    \label{phino}
\end{equation}
where $ \phi(n,m)
    =
    \prod_{j=m}^{n-1}
    \big(
    1-\beta(j)
    \big),
    \qquad m<n,$
and $\phi(n,n)=1.$ Note that, by definition, $\phi(n,m)\leq 1$ for all $m<n$. Further, for all $n$ and $k$, we have
$ \phi(n,k)
    =
    \phi(n,k+1)
    \big(
    1-\beta(k)
    \big).$ 
Consequently, $  \phi(n,k+1)-\phi(n,k) =
    \beta(k)\phi(n,k+1).$
Therefore, by a telescoping sum argument,
$ \sum_{k=0}^{n-1}
    \beta(k)\phi(n,k+1)
    =
    \phi(n,n)-\phi(n,0)
    <1.$

Hence, from \eqref{phino}, we obtain
\begin{equation}
    \norm{y_n}
    \leq
    \norm{y_0}
    +
    C
    +
    \norm{
    \sum_{k=0}^{n-1}
    \phi(n,k+1)\beta(k)M_{k+1}
    }.
    \label{ynams}
\end{equation}

We now show that the sequence $Z_n
=
\sum_{k=0}^{n-1}
\phi(n,k+1)\beta(k)M_{k+1}$
is almost surely bounded. Without loss of generality, assume that $\beta(0)=0$. Then,
\begin{equation*}
    \begin{split}
        Z_n
        &=
        \sum_{k=1}^{n-1}
        \phi(n,k+1)\beta(k)M_{k+1}
        \\
        &=
        \sum_{k=1}^{n-1}
        \phi(n,k+1)
        (S_k-S_{k-1})
        \\
        &=
        \phi(n,n)S_{n-1}
        +
        \sum_{k=1}^{n-2}
        S_k
        \big(
        \phi(n,k+1)-\phi(n,k+2)
        \big),
    \end{split}
\end{equation*}
where $S_n
=
\sum_{k=0}^{n}
\beta(k)M_{k+1}.$
From Corollary \ref{Corolloary assumption on noise}, the sequence $\{S_n\}$ is an $\mathcal{L}^2$-bounded martingale and therefore converges almost surely. Consequently, there exists a finite random variable $S^\ast$ such that $\norm{S_n}\leq S^\ast.$

Using this bound, we obtain
\begin{equation*}
    \begin{split}
        \norm{Z_n}
        \leq&
        \,
        S^\ast
        +
        S^\ast
        \sum_{k=1}^{n-2}
        \big(
        \phi(n,k+2)-\phi(n,k+1)
        \big)
        \leq
        \,
        2S^\ast.
    \end{split}
\end{equation*}

Finally, substituting the above estimate into \eqref{ynams}, we get

$ \norm{y_n}
    \leq
    \norm{y_0}
    +
    C
    +
    2S^\ast.$
This proves the corollary.
\end{proof}

\color{black}

\subsection{Asymptotic Pseudo-Trajectory of the Projected Stochastic Subgradient Method}
We introduce a continuous-time interpolation of the iterates ${x_n}$ via
\begin{equation}
    \begin{split}
        \Bar{x}(t) = x_n + (x_{n+1} -x_n) \frac{t - t(n)}{t(n+1)-t(n)}, \quad \forall \; t \in I_n
    \end{split}
    \label{CTI}
\end{equation}
where $t(n) = \sum\limits_{k=1}^{n} \alpha(k)$ and $I_n = [t_n,t_{n+1}]$.
The interpolated path $\bar{x}(t)$ is shown to be an asymptotic pseudo-trajectory of a projected dynamical system with disturbance input, enabling the use of Gronwall’s inequality to establish almost sure convergence. We begin with Proposition 5.3.5 of \cite{hiriart2013convex}.

\begin{proposition}
    For any given $x \in \mathbb{R}^d$ and $v \in \mathbb{R}^d$, the following holds: 
    \begin{equation*}
        \lim\limits_{t \downarrow 0} \frac{\mathcal{P}_\mathcal{X}(x+tv) - x}{t} = \mathcal{P}_{\mathcal{T}_\mathcal{X}(x)} (v).
    \end{equation*}
    \label{PDS}
\end{proposition}
Using this result, the update in \eqref{mirror descent} can be written as
\begin{equation*}
    \begin{split}
        x_{n+1} & = \mathcal{P}_\mathcal{X} (x_n -\alpha(n) y_n)
        \\ & = x_n + \alpha(n)  \mathcal{P}_{\mathcal{T}_\mathcal{X}(x_n)} (-y_n) + o(\alpha(n)) 
        \\ & = x_n - \alpha (n) ( y_n + \eta_n)  + o(\alpha(n))
    \end{split}
\end{equation*}

The equality above follows from Moreau’s Decomposition Theorem (Theorem 3.2.5 in \cite{hiriart2013convex}), which implies that 
\\ $ \eta_n = \mathcal{P}_{\mathcal{N}_\mathcal{X}(x_n)} (-y_n)$.
Since the normal cone is a closed convex cone and hence $0 \in \mathcal{N}_\mathcal{X}(x_n)$, it follows that  $ \norm{\eta_n + y_n} \leq \norm{y_n}$.
Applying the triangle inequality yields $ \norm{\eta_n} \leq 2 \norm{y_n}.$

The discussion so far culminates in the following Proposition. 

\begin{proposition}
   The  iterates in \eqref{mirror descent} can be equivalently expressed as a two time-scale stochastic recursive inclusion:
\begin{equation*}
        \begin{split}
           & x_{n+1} -x_n   \in  \alpha(n) H_1(x_n,y_n)
            \\ & y_{n+1} - y_n - \beta(n) M_{n+1}  \in \beta(n) H_2 (x_n,y_n)
        \end{split}
    \end{equation*}
    where the set-valued maps $H_1$ and $H_2$ are defined as
    \begin{equation*}
        \begin{split}
            H_1(x,y)  & = -(y+ \Hat{\mathcal{N}}_\mathcal{X} (x)) 
            \\ H_2(x,y) & = -( y- \partial  f(x) ) + B(0,r(\lambda)) \; \; \text{and}
        \end{split}
    \end{equation*}
    
    \begin{equation*}
        \Hat{\mathcal{N}}_\mathcal{X} (x)  = \{ \eta \in \mathcal{N}_\mathcal{X}(x) \; \; | \; \; \norm{\eta} \leq 2 \norm{y} \} .
    \end{equation*}
    \label{Proposition 2}
\end{proposition}
\begin{remark}
 Proposition \ref{Proposition 2} shows that the iterates follow a two time-scale stochastic recursive inclusion given in as defined in \cite{yaji2020stochastic}.  \color{black} 
One key advantage of the proposed two-timescale algorithm is established in Proposition \ref{Proposition 2}. Using the two-timescale stochastic recursive inclusion framework of \cite{yaji2020stochastic}, we show that the fast timescale iterates $y_n$ asymptotically track a subgradient of $f$ at $x_n$. Consequently, asymptotically $y_n$ may be replaced by some $g_n \in \partial f(x_n)$, and the slow time scale dynamics can then be analyzed via the stability of a differential inclusion.

An important consequence of the two-timescale structure is that the slow timescale dynamics contains no independent noise term, enabling a direct application of the framework in \cite{yaji2020stochastic}. In contrast, for the standard single timescale stochastic subgradient method over a general constrained set, the nonlinear projection operator $\mathcal{P}_{\mathcal{T}_\mathcal{X}(x)}(\cdot)$ prevents the iterates from being expressed in the standard stochastic approximation form, thereby precluding a direct application of the analysis in \cite{doi:10.1137/S0363012904439301}.
\label{Remark-two time scale}
\end{remark}
\begin{remark}
For a diminishing smoothing parameter sequence $\{\lambda_n\}$, the bias term $r(\lambda_n)$ is asymptotically negligible since $\lim_{n\to\infty} r(\lambda_n)=0$. The justification follows the argument on page 18 of \cite{borkar2008stochastic}. Hence, most of the analysis is carried out for a constant smoothing parameter $\lambda$. Nevertheless, Subsection \ref{subsectionalmsot} shows that the almost sure convergence result for diminishing smoothing parameters follows as a special case of the constant smoothing parameter setting.
\end{remark}
 \color{black}
The strategy of the remaining portion  is to apply Theorem 5.8 and Theorem 5.9 from \cite{yaji2020stochastic} to the two-timescale stochastic approximation iterates given in Proposition \ref{Proposition 2}.  For that we need to verify the Assumptions A1-A8 in \cite{yaji2020stochastic}. Corollary \ref{Corolloary assumption on noise} confirms that the martingale difference sequence ${M_n}$ satisfies A6, while the next lemma establishes the Marchaud property of $H_1$ and $H_2$ (A1 and A2).

\begin{lemma}
    The set-valued maps $H_1$ and $H_2$   are Marchaud.
    \label{Marchaud}
\end{lemma}

\begin{proof}
First, we show that $H_1$ is a Marchaud map. 

\textbf{Claim-1: For any $(x,y) \in \mathbb{R}^d$,  $H_1(x,y)$ is convex and compact.}

We first show that $\Hat{\mathcal{N}}_\mathcal{X}$ is convex. 
Let $ \; \eta_1, \eta_2 \in \Hat{\mathcal{N}}_\mathcal{X}(x)$   and $\theta \in [ 0,1]$. Since $\mathcal{N}_\mathcal{X}(x)$ is convex, we have $ \theta \eta_1 + (1-\theta) \eta_2 \in \mathcal{N}_\mathcal{X} (x)$.
Moreover, by the triangle inequality,
\begin{equation*}
    \norm{\theta \eta_1 + (1-\theta) \eta_2} \leq \theta \norm{\eta_1} + (1- \theta) \norm{\eta_2} \leq 2 \norm{y},
\end{equation*}
which proves that $\Hat{\mathcal{N}}_\mathcal{X}(x)$ is convex. Since the sum of convex sets is convex, it follows that $H_1(x,y)$ is also convex for each $(x,y) \in \mathbb{R}^{2d}$.
Note that since $\mathcal{N}_\mathcal{X}(x)$ is closed and for each $y$ $\Hat{\mathcal{N}}_\mathcal{X}(x)$ is bounded, it follows that $\Hat{\mathcal{N}}_\mathcal{X}(x)$ is  compact. Consequently, $H_1(x,y)$ is also compact.

\textbf{Claim-2: There exists a constant $\mathrm{K} > 0$ such that 
        \begin{equation*}
           \mathbf{ \sup\limits_{x' \in H_1(x,y)} \norm{x'} \leq \mathrm{K} (1 + \norm{x} + \norm{y}), \quad (x,y) \in \mathbb{R}^d.}
        \end{equation*}}

        Consider any $x' \in H_1(x,y)$. By definition,  $\exists \eta \; \in \Hat{\mathcal{N}}_\mathcal{X}(x)$ such that $  x' = -y -\eta$.
Taking norms and applying the triangle inequality yields
 $\norm{x'} \leq \norm{y} + \norm{\eta} \leq 3 \norm{y}$.
This establishes the claim. 

\textbf{Claim-3: $H_1$ is upper semicontinuous}

Let $\{ (x_n ,y_n) \} \subseteq \mathcal{X} \times \mathbb{R}^d$ be such that $\{ (x_n ,y_n) \}  \to (x,y)$,  and suppose $z_n \in H_1(x_n,y_n)$, $\forall n$, with $z_n \to z$. We aim to show that $z \in H_1(x,y)$.

By definition of $H_1$, for each $n$, $\exists \; \eta_n \in \Hat{\mathcal{N}}_\mathcal{X}(x_n)$ such that $ z_n = - y_n - \eta_n$.
Since $y_n \to y$ and $z_n \to z$, it follows that $\eta_n \to \eta = - y - z$. 

The normal cone mapping $\mathcal{N}_\mathcal{X}(x)$
 is upper semicontinuous, and $\eta_n \in \mathcal{N}_\mathcal{X}(x_n)$, so the limit $\eta \in \mathcal{N}_\mathcal{X}(x)$.

Moreover, by the definition of $\Hat{\mathcal{N}}_\mathcal{X}(x)$, we have $\norm{\eta_n} \leq 2 \norm{y_n}$.

Taking the limit as $n \to \infty$, we obtain $\norm{\eta} \leq 2 \norm{y},$

which implies $\eta \in \Hat{\mathcal{N}}_\mathcal{X}(x)$. 

Therefore, $  z = -y-\eta \in H_1(x,y).$
Thus the set-valued map $H_1$ is Marchaud. 
Next, we show that $H_2$ is Marchaud.

\textbf{Claim-4: For every $(x,y) \in \mathcal{X} \times \mathbb{R}^d$, the set $H_2(x,y)$ is convex and compact.}

Recall $  H_2(x,y)  = -( y- \partial  f(x) ) + B(0,r(\lambda))$.
Since the Clarke subdifferential $\partial f(x)$
 is a nonempty, convex, and compact subset of $\mathbb{R}^d$ for each 
$x$, and since the Minkowski sums of convex and  compact sets remain convex and compact, it follows that
 $H_2(x,y)$ is convex and compact for every $(x,y) \in \mathcal{X} \times \mathbb{R}^d$.

 \textbf{Claim-5: There exists a constant $\mathrm{K} > 0$ such that 
        \begin{equation*}
           \mathbf{ \sup\limits_{z \in H_2(x,y)} \norm{z} \leq \mathrm{K} (1 + \norm{x} + \norm{y}), \quad (x,y) \in \mathbb{R}^d.}
        \end{equation*}}

        Since $\mathcal{X}$
  is compact and $f$ is locally Lipschitz, the Clarke subdifferential 
$\partial f(x)$ is uniformly bounded on 
$\mathcal{X}$. Thus, there exists a constant 
$G> 0$ such that $ \norm{g} \leq G, \; \; \forall \; g \in \partial f(x) \; \; \text{and} \; \; \forall \; \; x \in \mathcal{X}.$

Let $z \in H_2(x,y)$. Then by definition,  $ z = -(y-g) + \zeta$

for some $g \in \partial f(x)$, $\zeta \in B (0,r(\lambda))$. Applying the triangle inequality $ \norm{z} \leq \norm{y} + G + r(\lambda).$
        This completes the proof of the claim.

        \textbf{Claim-6: $H_2$ is upper semicontinuous.}

        Let $\{ (x_n ,y_n) \} \subseteq \mathcal{X} \times \mathbb{R}^d$ be such that $\{ (x_n ,y_n) \}  \to (x,y)$,  and suppose $z_n \in H_2(x_n,y_n)$, $\forall n$, with $z_n \to z$. We aim to show that $z \in H_2(x,y)$.
By definition of $H_2$, for each $n$, $\exists \; g_n \in \partial f(x_n)$ and $h_n \in B(0,r(\lambda))$ such that $ z_n = - y_n + g_n - h_n$.
Since  $\{h_n\} \subseteq B (0, r(\lambda))$, it  is bounded. Thus, without loss of generality, we may assume (by passing to a subsequence if necessary) that $h_n \to h \in B(0,r(\lambda))$.  It then follows that $z_n \to z$, $y_n \to y$, and $h_n \to h$ and thus 
\begin{equation*}
    g_n = z_n + y_n + h_n \to z+y+h =: g
\end{equation*}
We now show that $g \in \partial f(x)$. 
\color{black}
Since each $g_n \in \partial f(x_n)$, $\forall \; v \in \mathbb{R}^d$ we have $\inprod{g_n}{v} \leq f^0(x_n,v) $.
Taking the upper limit on both sides yields:
\begin{equation*}
    \begin{split}
         \inprod{g}{v} & \leq \limsup\limits_{n \to \infty} f^0(x_n,v) \leq f^0(x,v), \quad \forall \; v \in \mathbb{R}^d. 
    \end{split}
\end{equation*}
The last inequality holds true in view that $f^0(x,v)$ is upper semicontinuous with respect to $x$ (Proposition 2.1 from \cite{clarke1990optimization}).
 This implies $g \in \partial f(x)$.
\color{black}
Thus,  we obtain $ z = - y + g -h \in H_2(x,y)$.

Hence, it follows that $H_2$ is upper semi-continuous. 
\end{proof}
Since $H_1$ and $H_2$ are Marchaud maps (Lemma~\ref{Marchaud}) and the step-size conditions in Assumption~\ref{ass-stepsize} hold, the two time-scale stochastic approximation framework applies to establish almost sure convergence of the iterates in \eqref{mirror descent}. 
\subsubsection{Analysis of Fast Time Scale}
We now analyze the asymptotic behavior of the fast time-scale iterates $y_n$, $n\geq 0$, treating $x_n$ as fixed at some $x$. The limiting behavior of $y_n$ is governed by the following differential inclusion:
\begin{equation}
     \begin{split}
        \Dot{y}(t) & \in H_2(x,y(t)) \quad 
         \text{such that} \; \; y(0)  = y_0 \in \mathbb{R}^{d}, 
        \end{split}
        \label{fast time scale1}
\end{equation}
where $  H_2(x,y) = -y  + \partial f(x) + B(0,r(\lambda))$. 
\begin{lemma}
Let $G_x = \partial f(x) + B (0,r(\lambda))$. Then 
\begin{enumerate}
    \item[i.] The equilibrium points of \eqref{fast time scale1} coincide with $G_x$.
    \item[ii.] Every point in $G_x$ is Lyapunov stable.
    \item[iii] Every solution of \eqref{fast time scale1} converges asymptotically to $G_x$.
\end{enumerate}
\label{lema fast time scale}
\end{lemma}

\begin{proof}
    It is straightforward to verify that $0 \in H_2 (x,y)$ if and only if $y \in G_x$. Therefore, every point $ y \in G_x$ is an equilibrium point of the differential inclusion \eqref{fast time scale1}.
To establish Lyapunov stability of the equilibrium set $G_x$, consider the candidate Lyapunov function: $  V(y) = \frac{1}{2} \dist (y,G_x)^2$
where $\dist (y,G_x) = \min\limits_{g \in G_x} \norm{y-g}$. Since $G_x$ is convex and compact, the projection $\mathcal{P}_{G_x}(y)$ is unique, and $V(y)$ is continuously differentiable. Moreover the gradient of $V$ is given by
$\nabla V(y) = (y- \mathcal{P}_{G_x} (y))$.
 To analyze the set-valued dynamics, consider the Lie derivative of 
$V$ along the solutions of the differential inclusion, defined as (cf. \cite{Cortes}):
 \begin{equation*}
         \begin{split}
             \mathcal{L} V (y) = \{ a \; | \; a = \nabla V(y)^\top \nu \; \; \text{where} \; \; \nu \in - y+ G_x \}
         \end{split}
     \end{equation*}
Let $a \in  \mathcal{L} V (y) $  Then, $\exists \; g \in G_x$ such that 
\begin{equation*}
     \begin{split}
         a & = \nabla V(y)^\top (-y + g)
         =  (y- \mathcal{P}_{G_x} (y))^\top (g-y)
             \\  & =  - \norm{y- \mathcal{P}_{G_x} (y)}^2 + (y- \mathcal{P}_{G_x} (y))^\top (g- \mathcal{P}_{G_x} (y))
             \\ & \leq - \norm{y- \mathcal{P}_{G_x} (y)}^2 < 0 \quad \text{when} \; \quad y \notin G_x .
         \end{split}
     \end{equation*}
The inequality in the last line follows from the fact that
$ (y- \mathcal{P}_{G_x} (y))^\top (g- \mathcal{P}_{G_x} (y)) \leq 0$
which holds by the Theorem 3.1.1 in \cite{hiriart2013convex}.
Therefore,
 \begin{equation*}
         \sup\limits_{a \in \mathcal{L} V(y)} \quad a \begin{cases}
             & < 0 \; \; \text{if} \; y \notin G_x
             \\ & = 0 \; \; \text{if} \; y \in G_x. 
         \end{cases}  
     \end{equation*}

     This shows that 
$V$ is a strict Lyapunov function for the differential inclusion \eqref{fast time scale1}. Therefore, by Theorem 6.2 of \cite{goebel2024set}, the set $G_x$	
  is globally asymptotically stable for \eqref{fast time scale1}. 
\end{proof} 

In Lemma~\ref{lema fast time scale}, we established that for each fixed $x$, the set $G_x = \partial f(x) + B (0,r(\lambda))$ serves as the global attractor for \eqref{fast time scale1}.

\begin{lemma}
    The set-valued mapping $x \rightrightarrows G_x$ is a Marchaud map. 
\end{lemma}

\begin{proof}
   The proof follows along the same lines as the argument used to show that $H_2$ is a Marchaud map. Hence, for the sake of brevity, we omit the details.
\end{proof}

\subsubsection{ Analysis of Slow Time Scale}
To analyze the asymptotic behavior of the slow time-scale iterates $x_n$, we define the associated limiting set-valued map $\Hat{H}_1 : \mathcal{X} \rightrightarrows \mathbb{R}^d$ as 
\begin{equation*}
\begin{split}
    \Hat{H}_1 (x) & = \cup_{y \in G_x} H_1(x,y)
    \\ & = - \partial f(x) - \Hat{\mathcal{N}}_\mathcal{X}(x) + B(0, r(\lambda)).
    \end{split}
\end{equation*}
The next lemma is a consequence of Lemma 13 of \cite{yaji2020stochastic}. 

\begin{lemma}
    The set-valued mapping $x \rightrightarrows \Hat{H}_1(x)$ is  Marchaud. 
\end{lemma}

Having verified all the assumptions in \cite{yaji2020stochastic}, we can now invoke Theorem 5.8 from \cite{yaji2020stochastic} to establish that the continuous time interpolation $\Bar{x}(t)$ of ${x_n}, n\geq 0$, is an asymptotic pseudo-trajectory of the projected differential inclusion, as formalized in the next proposition.

\begin{proposition}
      Let $\Bar{x}(t)$ denote the continuous-time interpolation of the iterates generated by the projected subgradient method as given in \eqref{CTI}. Then $\Bar{x}(t)$ is an asymptotic pseudotrajectory (APT) of differential inclusion \begin{equation} \begin{split} \Dot{x} (t) \in - (\partial_x f(x) + \Hat{\mathcal{N}}_\mathcal{X}(x) + B(0,r(\lambda)).
    \end{split}
    \label{APT normal1}
\end{equation}
    
    That is, for any  $T > 0$, 
     \begin{equation*}
\lim\limits_{t \to \infty}  \sup\limits_{0 \leq s \leq T} \norm{\Bar{x}(t+s)-x(s)} = 0,
\end{equation*}
where $x(s)$ denotes the solution of the differential inclusion \eqref{APT normal1} with initial condition $x(0) = \Bar{x}(t)$.
\label{AST1p}
\end{proposition}

\subsection{Almost Sure Convergence of Iterates }
\label{subsectionalmsot}
In the final part, we use Proposition~\ref{AST1p} to study the asymptotic behavior of ${x_n}$, beginning with a lemma on the stationarity of local minimizers.
\begin{lemma}
    Let $x^\ast$ be a local minimum of the optimization problem \eqref{eq:main_problem}. Then $x^\ast$ belongs to the set of stationary points $\mathcal{S}$, defined as
    \begin{equation*}
    \mathcal{S} = \{ x \in \mathcal{X} \; \; | \exists \;  \zeta \in \partial f(x) \; \text{s.t.} \; \; \inprod{\zeta}{y-x} \geq 0 \; \; \forall \; y \in \mathcal{X} \}. 
\end{equation*}
   Moreover, any point in $\mathcal{S}$ is a Clarke stationary point.
\end{lemma}

\begin{proof}
This proof follows directly from Theorem 8.15 in \cite{rockafellar1998variational}. Since $f$ is Lipschitz continuous, its horizon subdifferential satisfies $\partial^\infty f(x)=\{0\}$ by Theorem 9.13 of \cite{rockafellar1998variational}. Hence, the required constraint qualification holds, and the result follows directly from Theorem 8.15.
\end{proof}

We have established in Proposition \ref{AST1p} that the continuous-time interpolation of ${x_n,n\geq 0},$ behaves as an asymptotic pseudo-trajectory of the projected differential inclusion with a disturbance term. We now show that, in the absence of disturbance (i.e., when $r(\lambda) = 0$), every Carathéodory solution of the projected differential inclusion converges asymptotically to the set $\mathcal{S}$.

\begin{theorem}
Consider the projected differential inclusion
\begin{equation}
    \Dot{x}(t)
    \in
    \mathcal{P}_{\mathcal{T}_{\mathcal{X}}(x(t))}
    \bigl(-\partial f(x(t))\bigr).
    \label{PDStheorem}
\end{equation}

Then, any Carath\'eodory solution of \eqref{APT normal1} with $r(\lambda)=0$ is also a Carath\'eodory solution of \eqref{PDStheorem}, and vice versa. Moreover, the set of equilibrium points of \eqref{PDStheorem} coincides with the set $\mathcal{S}$.

Furthermore, every equilibrium point is Lyapunov stable, and every Carath\'eodory solution of \eqref{PDStheorem} converges to the set $\mathcal{S}$.
\label{stability1}
\end{theorem}
\begin{proof}
\color{black}
The first statement, namely that every Carath\'eodory solution of \eqref{APT normal1} is also a Carath\'eodory solution of \eqref{PDStheorem}, and vice versa, can be proved using arguments analogous to those in Corollary 5.5 of \cite{Doffler}. Hence, for the sake of brevity, we omit the details.
\color{black}

    Suppose  $ 0 \in \mathcal{P}_{\mathcal{T}_\mathcal{X}(x)} (- \partial f(x))$, then $\exists \; g \in \partial f(x) $ such that $ 0 = \mathcal{P}_{\mathcal{T}_\mathcal{X}(x)} (- g)$. 
   By Moreau’s decomposition theorem, we have  $ -g = \mathcal{P}_{\mathcal{N}_\mathcal{X}(x)} (- g)$
which implies $- g \in \mathcal{N}_\mathcal{X}(x)$. 
By the definition of the normal cone, it follows that
\\ $ \inprod{- g}{y-x} \leq 0, \; \;  \forall \; y \in \mathcal{X}, $
which shows that $x \in \mathcal{S}$. We can prove the converse in exactly a similar way. 
We now show that any Carathéodory solution of the projected differential inclusion converges to the set $\mathcal{S}$. To this end, consider the Lyapunov function
\begin{equation*}
    V(x) = f(x)  - f^\ast,
\end{equation*}
where $f^\ast = \min\limits_{x \in \mathcal{X}} f(x)$. Let $  F(x) = \mathcal{P}_{\mathcal{T}_\mathcal{X}(x)} (- \partial f(x)).$
We consider the set-valued Lie derivative of 
$V$ with respect to $F$, defined as 
 \begin{equation*}
        \Tilde{L}_F V(x) = \{a \in \mathbb{R} \;  | \;  \exists \;  v \in F(x) \;  \text{s.t.} \;  a = \zeta^\top v, \;  \forall \; \zeta \in \partial f(x) \}.
    \end{equation*}
    Let $a \in   \Tilde{L}_F V(x)$. Then, by definition,  $\exists \; \nu \in F(x)$, such that 
    $   a = \zeta^\top \nu, \; \; \forall \; \zeta \in \partial f(x)$.
    Since $\nu \in F(x)$,  there exists $g \in \partial f(x)$ such that 
    \begin{equation*}
    \begin{split}
         \nu & = \mathcal{P}_{\mathcal{T}_\mathcal{X}(x)} (-g)
          = \argmin\limits_{\nu_1 \in \mathcal{T}_\mathcal{X}(x)} \; \; \norm{\nu_1+ g}^2.
         \end{split}
    \end{equation*}

Since tangent cone is a closed convex cone, we have $0 \in \mathcal{T}_\mathcal{X}(x)$, and hence
$  \norm{\nu + g}^2 \leq \norm{g}^2.$
This yields $a = \inprod{\nu}{g} \leq - \frac{1}{2} \norm{\nu}^2.$
In other words, $ \sup\limits_{a \in \Tilde{L}_F V(x)}  \; \; \{a \}  \; \leq 0.$
By invoking Proposition 10 of \cite{Cortes}, it follows that $  \frac{d}{dt} f(x(t)) \leq 0 \; \; \quad \text{for almost all} \; t \in [0,\infty),$
which implies that the function 
$f(x(t))$ is non-increasing almost everywhere along Carathéodory solutions.

Since $\mathcal{X}$ is assumed to be compact and positively invariant under the dynamics, we apply Theorem 4 of \cite{Cortes} to conclude that every Carathéodory solution of \eqref{PDStheorem} with initial condition in  $\mathcal{X}$ asymptotically converges to the largest weakly invariant set contained in  
\\ $\mathcal{X} \cap \{x \in \mathbb{R}^d \; \; |  \; \;  0 \in \Tilde{L}_F V(x) \}.$

Now, if $0 \in  \Tilde{L}_F V(x)$, then there exists $  \nu \in F(x)$ such that 
\\ $ 0 = \nu^\top \zeta,  \; \quad \forall \; \zeta \in \partial f(x)$. 
Since $\nu \in F(x)$,  by definition, there exists $g \in \partial f(x)$ such that 
\begin{equation}
    \nu = \mathcal{P}_{\mathcal{T}_\mathcal{X}(x)} (-g). 
\end{equation}

Using Moreau’s decomposition theorem, this implies $- g \in \mathcal{N}_\mathcal{X}(x),$
which shows that $ x \in \mathcal{S}$. Therefore, the largest invariant set contained in $ \mathcal{X} \cap \{x \in \mathbb{R}^d \; \; |  \; \;  0 \in \Tilde{L}_F V(x) \}$ is a subset of $\mathcal{S}$, and the claim follows.   
\end{proof}
\color{black}
In the remaining part of the paper, we establish the almost sure convergence of the iterates for two different choices of the smoothing parameter. First, for a constant smoothing parameter $\lambda$, we show that the iterates almost surely converge to a neighborhood of the set of Clarke stationary points $\mathcal{S}$. Next, for a diminishing smoothing parameter sequence $\{\lambda_n\}$ satisfying Assumption \ref{Assumptionregardingsmoothing}, we show that the iterates almost surely converge to the set $\mathcal{S}$ itself.
\color{black}
\subsubsection{Almost Sure Convergence for Constant Smoothing Parameter}
The next theorem combines Proposition \ref{AST1p} and Theorem \ref{stability1} with robust stability analysis to show that ${x_n,n\geq 0},$ almost surely converge to a neighborhood of $\mathcal{S}$, whose size is governed by $\lambda$.
\begin{theorem}
  \textcolor{black}{ Let Assumptions \ref{assumpvari}, \ref{ass-stepsize} and \ref{monotone} hold, and let $\epsilon>0$ be arbitrary.} Then there exists $\lambda_0 > 0$ and $n_0 \in \mathbb{N}$ such that for all $\lambda \leq \lambda_0$, 
 the iterates $x_n, n\geq n_0$,	
  lie within the $\epsilon$-neighborhood of the set of stationary points $\mathcal{S}$, i.e., $x_n \in \mathrm{N}_\epsilon (\mathcal{S})$.
  \label{main Theorem1}
\end{theorem}

\begin{proof}
    Let $\Bar{x}(t)$ denote the continuous-time interpolation of the iterates $x_n,n\geq 0$, as defined in \eqref{CTI}. From Proposition \ref{APT normal1}, it follows that for any $T > 0$, we have 
    \begin{equation}
\lim\limits_{t \to \infty}  \sup\limits_{0 \leq s \leq T} \norm{\Bar{x}(t+s)-x_t(s)} = 0,
\label{AST1}
\end{equation}
where $x_t(s)$ is the solution of the differential inclusion 
 \begin{equation}
    \begin{split}
        \Dot{x}(t) \in - (\partial f(x) + \mathcal{N}_\mathcal{X}(x) + B(0,r(\lambda))),
    \end{split}
    \label{APT normal11}
\end{equation}
with initial condition $x_t(0) = \Bar{x}(t)$. 

The goal of the proof is to show that $\exists \; t^0$ such that $\forall \; t \geq t^0$, we have $\Bar{x}(t) \in \mathrm{N}_\epsilon (\mathcal{S})$. From the construction of $\Bar{x}(t)$ in \eqref{CTI}, this in turn establishes the conclusion of Theorem~\ref{main Theorem1}.

From \eqref{AST1}, $\exists \; t_0 > 0$ such that $\forall \; t \geq t_0$, we have 

\begin{equation}
    \sup\limits_{0 \leq s \leq T} \norm{\Bar{x}(t+s)-x_t(s)} \leq \frac{\epsilon}{3}.
    \label{sup}
\end{equation}

Now, consider the solution $x_t^1 (s),s\geq 0$ of the differential inclusion \eqref{APT normal11} with the same initial condition $x_t^1(0) = \Bar{x}(t)$,  but with the perturbation term removed, i.e., $r(\lambda) = 0$. Then, by Theorem~\ref{stability1}, $\exists \; T_0>0$ such that $x_t^1(T_0) \in \mathrm{N}_{\frac{\epsilon}{3}} (\mathcal{S})$. 

Moreover, since the constraint set $\mathcal{X}$ is compact, this time $T_0$ can be chosen uniformly for all initial points $\Bar{x}(t) \in \mathcal{X}$, and hence is independent of $t$. 

Consider the following term.
\begin{equation*}
       \begin{split}
          &   \frac{d}{ds}(\norm{x_t(s)-x_t^1(s)}^2)
            =  ( (x_t(s)-x_t^1(s))^\top (\Dot{x}_t(s)-\Dot{x}_t^1(s))) 
           \\ \overset{(a)}= &  ( (x_t(s)-x_t^1(s))^\top (-g_t(s) - \eta_t(s) + b(s) + g_t^1(s) + \eta_t^1(s)))
           \\ \overset{(b)}\leq & L \norm{x_t(s)-x_t^1(s)}^2 + r(\lambda) D.
       \end{split}
       \label{10000}
   \end{equation*}
In step (a), we use $g_t(s) \in \partial f(x_t(s))$, $g_t^1(s) \in \partial f(x_t^1(s))$, $\eta_t(s) \in \mathcal{N}\mathcal{X}(x_t(s))$, and $\eta_t^1(s) \in \mathcal{N}\mathcal{X}(x_t^1(s))$, with the perturbation $b(s)$ satisfying $|b(s)| \leq r(\lambda)$. Step (b) follows from Assumption \ref{monotone} and the normal cone definition, and the constraint set has diameter $D$. Thus, for any $0 \leq s \leq T_0$,
    \begin{equation*}
        \norm{x_t(s)-x_t^1(s)}^2 \leq L \int\limits_{0}^{s}  \norm{x_t(z)-x_t^1(z)}^2 dz + r(\lambda) D T_0.
    \end{equation*}
Applying Grönwall’s inequality yields
\begin{equation*}
    \norm{x_t(s)-x_t^1(s)}^2 \leq r(\lambda) D T_0 \exp{ (L T_0)}.
\end{equation*}

Since $\lim\limits_{\lambda \to 0} r(\lambda) = 0$, it follows that  $\exists \; \lambda_0$ such that $\forall \; \lambda \leq \lambda_0$,  

\begin{equation*}
    \norm{x_t(s)-x_t^1(s)} \leq \frac{\epsilon}{3}.
\end{equation*}

From \eqref{sup}, it follows that $\exists \; t_0$ such that $\forall \; t \geq t_0$, 

\begin{equation*}
  \norm{  \Bar{x}(t+T_0) - x_t(T_0) } \leq \frac{\epsilon}{3}.
\end{equation*}

Moreover, we have already established that 
\begin{equation*}
    \norm{x_t(T_0) - x_t^1(T_0)} \leq \frac{\epsilon}{3}, \; \; \text{and} \; \;  x_t^1(T_0) \in \mathrm{N}_{\frac{\epsilon}{3}} (\mathcal{S}).
\end{equation*}

Combining these three facts via the triangle inequality yields that $\exists \; t_0>0$ such that $\forall \; t \geq t_0$ we have $\Bar{x}(t+T_0) \in \mathrm{N}_\epsilon (\mathcal{S})$.
\end{proof}

\begin{remark}
Theorem \ref{main Theorem1} extends the asymptotic convergence result of \cite{Ramaswami} from smooth to non-smooth objective functions by showing convergence of the iterates to a neighborhood of the set of Clarke stationary points. A natural next step is to identify noise conditions under which the iterates avoid saddle points, analogous to the results in \cite{pemantle1990nonconvergence} for standard stochastic approximation algorithms.
\end{remark}
\subsubsection{Almost Sure Convergence for Diminishing Smoothing Parameter}

\color{black}
\begin{theorem}
Under Assumptions \ref{assumpvari}--\ref{Assumptionregardingsmoothing}, the iterates $\{x_n\}$ almost surely converge to the set of Clarke stationary points $\mathcal{S}$.
\label{Thmdiminsihing}
\end{theorem}

\begin{proof}
The proof follows from the two-timescale stochastic approximation framework of \cite{yaji2020stochastic}. For diminishing $\lambda_n$, the Marchaud property of $H_1$ and $H_2$ in Proposition \ref{Proposition 2} follows by arguments similar to Lemma \ref{Marchaud}. Further, the summability of the noise sequence $\sum_{n\geq 0} \beta(n) M_{n+1}$ was established in Corollary \ref{Corolloary assumption on noise}. Hence, all assumptions of Theorem 5.9 in \cite{yaji2020stochastic} are satisfied.

We now analyze the fast and slow timescale iterates separately.

\textbf{Step 1: Fast Time Scale Iterates.}
We first treat the slow timescale iterates $x_n$ as quasi-static, i.e., $x_n=x$. In this case, the asymptotic behavior of the fast timescale iterates $\{y_n\}$ is governed by the differential inclusion: $\Dot{y}(t) \in -y(t) + \partial f(x).$
Here, the bias term $r(\lambda_n)$ is neglected since $r(\lambda_n)\to 0$, and its contribution does not affect the asymptotic analysis; see the discussion on page 18 of \cite{borkar2008stochastic}.

Hence, using arguments similar to those in Lemma \ref{lema fast time scale}, we conclude that for every fixed $x_n=x$, $ y_n \to G_x = \partial f(x).$

\textbf{Step 2: Slow Time Scale Iterates.}
Using the standard two-timescale stochastic approximation results from \cite{yaji2020stochastic}, the asymptotic behavior of the slow timescale iterates $\{x_n\}$ is described by the following differential inclusion:
\begin{equation}
    \begin{split}
        \Dot{x}(t)
        &\in \Hat{H}_1(x)
        =
        \bigcup_{y\in G_x} H_1(x,y)
        \\
        &=
        -\partial f(x)-\Hat{\mathcal{N}}_{\mathcal{X}}(x).
    \end{split}
    \label{sow}
\end{equation}

Finally, Theorem \ref{stability1} shows that every Carath\'eodory solution of \eqref{sow}, corresponding to any initial condition $x_0\in\mathcal{X}$, asymptotically converges to the set $\mathcal{S}$. Therefore, the conclusion follows directly from Theorem 5.9 of \cite{yaji2020stochastic}.
\end{proof}
\begin{remark}
Theorem \ref{Thmdiminsihing} establishes almost sure convergence to the set of Clarke stationary points $\mathcal{S}$ under $\lambda_n \to 0$ without requiring Assumption \ref{monotone}. In contrast, a constant smoothing parameter $\lambda>0$ yields uniformly bounded variance, allowing larger step sizes and often faster practical convergence (to a neighborhood of $\mathcal{S})$, but requires the additional one-sided Lipschitz assumption in Assumption \ref{monotone}. A similar bias--variance trade-off has been discussed in Chapter 4.4 of \cite{haykin2009neural} for continuously differentiable functions. However, a detailed investigation of this trade-off in our setting would require a comprehensive finite-time analysis of the proposed algorithm. Deriving such finite-time guarantees under the sole assumption of Lipschitz continuity, without imposing additional regularity conditions, is a highly nontrivial problem \cite{zhang2020complexity}. Consequently, the present manuscript focuses on asymptotic convergence properties. Developing finite-time performance bounds under stronger assumptions, such as weak convexity or other suitable regularity conditions, and using these bounds to characterize the bias--variance trade-off more precisely, constitutes an interesting direction for future research.
\end{remark}
\color{black}

\section{Numerical Simulation}

To empirically validate the proposed two time-scale zeroth-order projected stochastic subgradient method, we consider a robust parameter estimation problem modeled using a single-layer nonlinear network.

We first generate a standard Gaussian random vector of $d$-dimensional $x \sim \mathcal{N}(0,I_d)$ and define $x^\ast = \frac{x}{\norm{x}},$
so that $x^\ast$ lies inside the unit ball. Next, we generate the input feature vectors $\{u_i\}_{i=1}^N$, where each $u_i$ is independently sampled from the standard Gaussian distribution. The corresponding target values $\{y_i\}_{i=1}^N$ are generated according to the
\begin{equation*}
    y_i = \tanh\big((x^\ast)^\top u_i\big) + \zeta,
\end{equation*}
where $\zeta \sim \mathcal{N}(0,1)$ denotes the observation noise.

Given the dataset $\{(u_i,y_i)\}_{i=1}^N$, where $u_i \in \mathbb{R}^d$ and $y_i \in \mathbb{R}$, our objective is to minimize the empirical risk associated with the absolute deviation loss. The resulting optimization problem is
\begin{equation}
    \min_{x\in\mathcal{X}}
    f(x)
    =
    \frac{1}{N}
    \sum_{i=1}^N
    \left| y_i - \tanh(x^\top u_i) \right|,
    \label{eq:sim_objective}
\end{equation}
where $x \in \mathbb{R}^d$ is the optimization variable and the constraint set is given by
\begin{equation}
    \mathcal{X}
    =
    \left\{
    x \in \mathbb{R}^d
    \mid
    \norm{x}_2 \leq R
    \right\}.
\end{equation}

The objective function in \eqref{eq:sim_objective} satisfies the assumptions required in the proposed framework. In particular, the objective is non-convex and non-smooth since it is composed of the nonlinear function $\tanh(\cdot)$ and the non-smooth absolute value function $|\cdot|$. Moreover, it can be verified that $f(x)$ satisfies Assumption 4 of the manuscript.

To solve \eqref{eq:sim_objective} without access to explicit gradient information, we employ the proposed two time-scale zeroth-order projected stochastic subgradient method. At iteration $n \geq 0$, we sample a mini-batch $\mathcal{B}_n \subset \{1,\dots,N\}$ and define the stochastic oracle
\[
F(x,\zeta_n)
=
\frac{1}{|\mathcal{B}_n|}
\sum_{i\in\mathcal{B}_n}
\left|
y_i-\tanh(x^\top u_i)
\right|,
\]
where $\zeta_n$ represents the randomness induced by the mini-batch selection.

To approximate an element of the Clarke subdifferential, we use Gaussian smoothing. Let $\{U_n\}$ be a sequence of independent Gaussian random vectors with $U_n \sim \mathcal{N}(0,I_d)$, and let $\lambda>0$ denote the smoothing parameter. The approximated subgradient $\tilde{g}(n)$ is computed using the central difference estimator
\begin{equation}
    \tilde{g}(n)
    =
    \left[
    \frac{
    F(x_n+\lambda U_n,\zeta_n^1)
    -
    F(x_n-\lambda U_n,\zeta_n^2)
    }{2\lambda}
    \right]
    U_n.
    \label{eq:subgrad_approx}
\end{equation}

The updates for the proposed two time-scale stochastic subgradient method are given by
\begin{align}
    y_{n+1}
    &=
    y_n
    +
    \beta(n)
    \big(
    \tilde{g}(n)-y_n
    \big),
    \label{eq:fast_update}
    \\
    x_{n+1}
    &=
    \mathcal{P}_{\mathcal{X}}
    \big(
    x_n-\alpha(n)y_n
    \big).
    \label{eq:slow_update}
\end{align}

In the simulation setup, we choose the step-size sequences as
\begin{equation*}
    \beta(n)
    =
    \frac{100}{(n+1)^{0.6}},
    \qquad
    \alpha(n)
    =
    \frac{100}{(n+1)^{0.9}}.
\end{equation*}

In the following figures, we illustrate the behavior of the estimation error $\norm{x_n-x^\ast}$ as a function of the number of iterations for different choices of the smoothing parameter $\lambda$.

\begin{figure}[h!]
    \centering
    \includegraphics[width=0.8\linewidth]{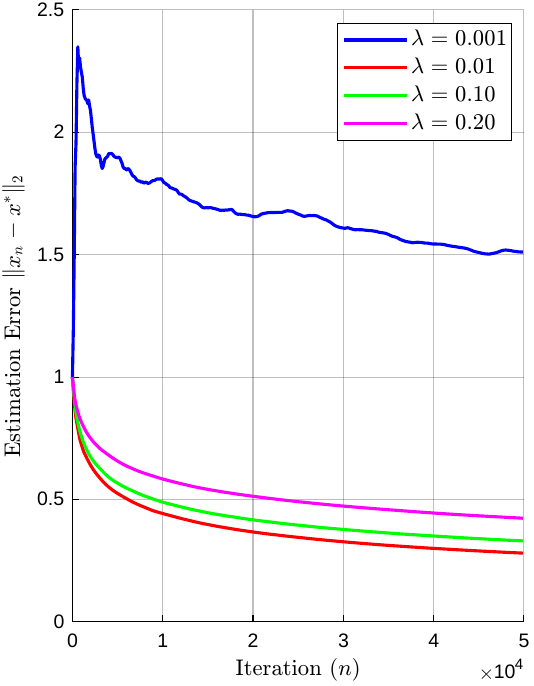}
    \caption{Performance of Two Time scale stochastic subgradient algorithm for different choices of $\lambda$.}
    \label{fig:placeholder}
\end{figure}

\begin{figure}[h!]
    \centering
    \includegraphics[width=0.8\linewidth]{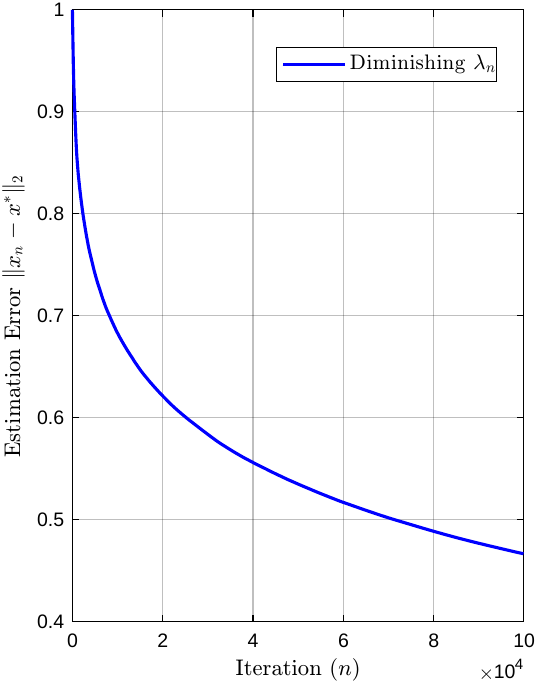}
    \caption{Performance of Two Time scale stochastic subgradient algorithm with diminishing $\lambda_n$.}
    \label{fig:placeholder1}
\end{figure}

From Figures~\ref{fig:placeholder} and \ref{fig:placeholder1}, we observe the performance of the two-timescale stochastic subgradient algorithm under both constant and diminishing smoothing parameters.

Figure~\ref{fig:placeholder} illustrates the behavior of the algorithm for different constant values of the smoothing parameter $\lambda$. We observe that reducing $\lambda$ generally decreases the estimation error $\|x_n-x^\ast\|$ as the iterations progress. However, $\lambda$ cannot be chosen arbitrarily small, since excessively small values may adversely affect the performance of the gradient estimator and the overall optimization process, as also evident from Figure~\ref{fig:placeholder}.

On the other hand, Figure~\ref{fig:placeholder1} considers a diminishing sequence $\lambda_n$ satisfying Assumption~\ref{Assumptionregardingsmoothing}. In this case, the optimization error gradually decreases as the smoothing bias vanishes over time, leading to improved asymptotic performance compared to the constant setting $\lambda$.

\section{Conclusion}
We studied stochastic optimization of non-smooth, non-convex objectives under constraints, given only noisy function evaluations. By combining Gaussian smoothing with a two time-scale stochastic approximation framework, we proposed a zeroth-order algorithm that ensures almost sure convergence without requiring explicit subgradient information. The results establish a rigorous foundation for zeroth-order methods in challenging optimization landscapes. Future work will include deriving finite-time guarantees and extending the approach to non-Euclidean geometries, particularly within the mirror descent framework.

\bibliographystyle{IEEEtran}
\bibliography{ref}
\end{document}